\documentclass[12pt]{amsart}
\usepackage{mathrsfs}

\usepackage{amssymb,amsmath,graphicx,color,textcomp, amsthm,bbm,bbold, enumerate,booktabs}
\usepackage{chngcntr}
\usepackage{apptools}
\AtAppendix{\counterwithin{theorem}{section}}
\definecolor{Red}{cmyk}{0,1,1,0}

\definecolor{verde}{cmyk}{1,0,1,0}

\definecolor{loka}{cmyk}{.5,0,1,.5}
\definecolor{azul}{cmyk}{1,1,0,0}

\numberwithin{equation}{section}

\newcommand{\be}{\begin{equation}}
\newcommand{\ee}{\end{equation}}

\newtheorem{theorem}{Theorem}

\newtheorem{definition}{Definition}

\begin{document}
\title{The Minkowski's inequality by means of a generalized fractional integral}
\author{J. Vanterler da C. Sousa$^1$}
\address{$^1$ Department of Applied Mathematics, Institute of Mathematics,
 Statistics and Scientific Computation, University of Campinas --
UNICAMP, rua S\'ergio Buarque de Holanda 651,
13083--859, Campinas SP, Brazil\newline
e-mail: {\itshape \texttt{ra160908@ime.unicamp.br, capelas@ime.unicamp.br }}}

\author{E. Capelas de Oliveira$^1$}

\begin{abstract}  We use the definition of a fractional integral, recently proposed by Katugampola, to establish a generalization of the reverse Minkowski's inequality. We show two new theorems associated with this inequality, as well as state and show other inequalities related to this fractional operator.

\end{abstract}
\maketitle


\section{Introduction}
Studies involving integral inequalities are important in several areas of science: mathematics, physics, engineering, among others, in particular we mention: initial value problem, linear transformation stability, integral-differential equations, and impulse equations \cite{BCAL,BDSP}.

The space of $p$-integrable functions $L^p(a,b)$ play a relevant role in the study of inequalities involving integrals and sums. Further, it is possible to extend this space of $p$-integrable functions, to the space of the measurable Lebesgue functions, denoted by $X_{c}^{p} (a,b)$, in which the space $L^{p}(a,b)$ is contained \cite{AHMJ}. Thus, new results involving integral inequalities have been possible and consequently, some applications have been made \cite{BCAL,BDSP}. We mention few of them, the inequalities of: Minkowski, Hölder, Hardy, Hermite-Hadamard, Jensen, among others \cite{HO,SEMD,SWK,BL,HO1,BP,OJC}.

On the other hand, non-integer order calculus, usually referred to as fractional calculus, is used to generalizes of integrals and derivatives, in particular integrals involving inequalities. There are many definitions of fractional integrals, for example: Riemann-Liouville, Hadamard, Liouville, Weyl, Erdéryi-Kober and Katugampola \cite {AHMJ,RHM,IP,UNT1}. Recently, Khalil et al. \cite{KRHA} and Adeljawad \cite{ABT}, introduced the local conformable fractional integrals and derivatives. From such fractional integrals, one obtains generalizations of the inequalities: Hadamard, Hermite-Hadamard, Opial, Gruss, Ostrowski, among others \cite{BMT,ADR,OMKH,CHFG,WJXM,SMBD,CHKU}.

Recently, Katugampola \cite{UNT} proposed a fractional integral unifying other well known ones: Riemann-Liouville, Hadamard, Weyl, Liouville and Erdélyi-Kober. Motivated by this formulation, we present a generalization of the reverse Minkowski's inequality \cite{SEM,DZO, CHUVP}, using the fractional integral introduced by Katugampola. We point out that studies in this direction, involving fractional integrals, are growing in several branches of mathematics \cite{OMKH,DRSS,YHK}.

The work is organized as follows: In section 2, we present the definition of the fractional integral, as well as its particular cases. We present the main theorems involving the reverse Minkowski's inequality, as well as the suitable spaces for such definitions. In section 3, our main result, we propose the reverse Minkowski's inequality using the fractional integral. In section 4, we discuss other inequalities involving this fractional integral. Concluding remarks close the article.

\section{Prelimiaries}

In this section, we present the reverse Minkowski's inequality theorem associated with the classical Riemann integral and its respective generalization via Riemann-Liouville and Hadamard fractional integrals. In addition, we present the fractional integral introduced by Katugampola, and we conclude with a theorem in order to recover particular cases.

Erhan et al. \cite{SEMD} address the inequalities of Hermite-Hadamard and reverse Minkowski for two functions $f$ and $g$ by means of the classical Riemann integral. On the other hand, Lazhar \cite{BL} also proposed a work related to the inequality involving integrals, that is, Hardy's inequality and the reverse Minkowski's inequality. Two theorems below were motivation for the works performed so far, via the Riemann-Liouville and Hadamard integrals, involving the reverse Minkowski's inequality.

\begin{definition} The space ${X}_{c}^{p}(a,b)$ $(c\in\mathbb{R}, 1\leq{p}\leq{\infty})$ consists of those complex-valued Lebesgue measurable functions $f$ on $(a,b)$, for which ${\lVert f \rVert}_{{X_{c}^{p}}}<\infty$ with
$${\lVert f \rVert}_{{X_{c}^{p}}}=\left(\int_{a}^{b}|x^{c}f(x)|^{p}\,\frac{dx}{x}\right)^{1/p}\; (1\leq{p}<{\infty})$$
and
$$\left\|f\right\|_{X_{c}^{\infty}}=\sup \textnormal{ess}_{x\in(a,b)}\,[x^{c}|f(x)|].$$
\end{definition}

In particular, when $c=1/p$ the space ${X}_{c}^{p}(a,b)$ coincides with the space $L^{p}(a,b)$ {\rm\cite{AHMJ}}.

\begin{theorem} Let $f,g\in L^{p}(a,b)$ be two positive functions, with $1\leq{p}\leq{\infty}$, $0<\displaystyle\int_{a}^{b}f^{p}\left(t\right) dt<\infty $ and $0<\displaystyle\int_{a}^{b}g^{p}\left( t\right) dt<\infty $. If $0<m\leq \displaystyle\frac{f\left( t\right) }{g\left( t\right) }\leq M$, for $m,M\in\mathbb{R}^{*}_{+}$ and $\forall t \in[a,b]$, then
\begin{equation}\label{A1}
\left( \int_{a}^{b}f^{p}\left( t\right) dt\right) ^{\frac{1}{p}}+\left(
\int_{a}^{b}g^{p}\left( t\right) dt\right) ^{\frac{1}{p}}\leq c_{1}\left(
\int_{a}^{b}\left( f^{p}+g^{p}\right) \left( t\right) dt\right) ^{\frac{1}{p}},
\end{equation}
with $c_{1}=\displaystyle\frac{M\left( m+1\right) +\left( M+1\right) }{\left( m+1\right)\left( M+1\right) }$ {\rm\cite{SEMD}}. 
\end{theorem}

\begin{theorem} Let $f,g\in L^{p}(a,b)$ be two positive functions, with $1\leq{p}\leq{\infty}$, $0<\displaystyle\int_{a}^{b}f^{p}\left( t\right) dt<\infty $ and $0<\displaystyle\int_{a}^{b}g^{p}\left( t\right) dt<\infty $. If $0<m\leq \displaystyle\frac{f\left( t\right) }{g\left( t\right) }\leq M$, for $m,M\in\mathbb{R}^{*}_{+}$ and $\forall t \in[a,b]$, then
\begin{equation}\label{A2}
\left( \int_{a}^{b}f^{p}\left( t\right) dt\right) ^{\frac{2}{p}}+\left(
\int_{a}^{b}g^{p}\left( t\right) dx\right) ^{\frac{2}{p}}\geq c_{2}\left(
\int_{a}^{b}f^{p}\left( t\right) dx\right) ^{\frac{1}{p}}\left(
\int_{a}^{b}g^{p}\left( t\right) dx\right) ^{\frac{1}{p}},
\end{equation}
with $c_{2}= \displaystyle\frac{\left( M+1\right) \left( m+1\right) }{M}-2 $ {\rm\cite{SEMD}}. 
\end{theorem}

We present the definitions of the fractional integrals that will be useful in the development of the article: Riemann-Liouville fractional integral, Hadamard integral, Erdélyi-Kober integral, Katugampola integral, Weyl integral and Liouville integral. 

\begin{definition} Let $[a,b]$ $(-\infty<a<b<\infty)$ be a finite interval on the real-axis $\mathbb{R}$. The Riemann-Liouville fractional integrals $(\mbox{left-sided and right-sided})$ of order $\alpha\in\mathbb{C}$, $\rm Re(\alpha)>0$, are defined by
\begin{equation}\label{A3}
J_{a^{+}}^{\alpha }f\left( x\right) :=\frac{1}{\Gamma \left( \alpha
\right) }\int_{a}^{x}\frac{f\left( t\right) }{\left( x-t\right) ^{1-\alpha }}dt, \text{ }x>a
\end{equation}
and
\begin{equation}\label{A4}
J_{b^{-}}^{\alpha }f\left( x\right) :=\frac{1}{\Gamma \left( \alpha \right) }\int_{x}^{b}\frac{f\left( t\right) }{\left( t-x\right) ^{1-\alpha }}dt, \text{ }x<b,
\end{equation}
respectively {\rm\cite{AHMJ,IP}}.
\end{definition}

\begin{definition} Let $(a,b)$ $(0\leq a<b<\infty)$ be a finite or infinite interval on the half-axis $\mathbb{R}^{+}$. The Hadamard fractional integrals $(\mbox{left-sided and right-sided})$ of order $\alpha\in\mathbb{C}$, $\rm Re(\alpha)>0$ of a real function $f\in L^{p}(a,b)$ are defined by
\begin{equation}\label{A5}
H_{a^{+}}^{\alpha }f\left( x\right) :=\frac{1}{\Gamma \left( \alpha
\right) }\int_{a}^{x}\left( \log \frac{x}{t}\right) ^{\alpha -1}\frac{%
f\left( t\right) }{t}dt, \text{ }a<x<b
\end{equation}
and
\begin{equation}\label{A6}
H_{b^{-}}^{\alpha }f\left( x\right) :=\frac{1}{\Gamma \left( \alpha
\right) }\int_{x}^{b}\left( \log \frac{t}{x}\right) ^{\alpha -1}\frac{%
f\left( t\right) }{t}dt\alpha, \text{ }a<x<b,
\end{equation}
respectively {\rm\cite{AHMJ,IP}}.
\end{definition}

\begin{definition} Let $(a,b)$ $(-\infty\leq a<b\leq\infty)$ be a finite or infinite interval or  half-axis $\mathbb{R}^{+}$. Also let $\rm Re(\alpha)>0$, $\sigma>0$ and $\eta\in\mathbb{C}$. The Erdélyi-Kober fractional integrals $(\mbox{left-sided and right-sided})$ of order $\alpha\in\mathbb{C}$ of a real function $f\in L^{p}(a,b)$ are defined by
\begin{equation}\label{A7}
I_{a^{+},\sigma ,\eta }^{\alpha }f\left( x\right) :=\frac{\sigma
x^{-\sigma \left( \alpha +\eta \right) }}{\Gamma \left( \alpha \right) }%
\int_{a}^{x}\frac{t^{\sigma \left( \eta +1\right) -1}}{\left( x^{\sigma
}-t^{\sigma }\right) ^{1-\alpha }}f\left( t\right) dt, \text{ } 0\leq a<x<b\leq\infty
\end{equation}
and
\begin{equation}\label{A8}
I_{b^{-},\sigma ,\eta }^{\alpha }f\left( x\right) :=\frac{\sigma
x^{\sigma \eta }}{\Gamma \left( \alpha \right) }\int_{x}^{b}\frac{t^{\sigma
\left( 1-\eta -\alpha \right) -1}}{\left( t^{\sigma }-x^{\sigma }\right)
^{1-\alpha }}f\left( t\right) dt, \text{ } 0\leq a<x<b\leq\infty,
\end{equation}
respectively {\rm\cite{AHMJ,IP}}.
\end{definition}

\begin{definition} Let $[a,b]\subset\mathbb{R}$ be a finite interval. Then the Katugampola fractional integrals $(\mbox{left-sided and right-sided})$ of order $\alpha\in\mathbb{C}$, $\rho>0$, $\rm Re(\alpha)>0$ of a real function $f\in X^{p}_{c}(a,b)$ are defined by
\begin{equation}\label{A9}
^{\rho }I_{a^{+}}^{\alpha }f\left( x\right) :=\frac{\rho ^{1-\alpha
}}{\Gamma \left( \alpha \right) }\int_{a}^{x}\frac{t^{\rho -1}}{\left(
x^{\rho }-t^{\rho }\right) ^{1-\alpha }}f\left( t\right) dt, \text{ }x>a
\end{equation}
and
\begin{equation}\label{10}
^{\rho }I_{b^{-}}^{\alpha }f\left( x\right) :=\frac{\rho ^{1-\alpha
}}{\Gamma \left( \alpha \right) }\int_{x}^{b}\frac{t^{\rho -1}}{\left(
t^{\rho }-x^{\rho }\right) ^{1-\alpha }}f\left( t\right) dt, \text{ } x<b,
\end{equation}
respectively {\rm\cite{UNT1}}.
\end{definition}

\begin{definition} The Weyl fractional integrals of order $\alpha\in\mathbb{C}$, $\rm Re(\alpha)>0$ of a real function $f$ locally integrated into $(-\infty, \infty)$ being $-\infty\leq x\leq \infty$ are defined by
\begin{equation}\label{A11}
_{x}W_{\infty }^{\alpha }=_{x}I_{\infty }^{\alpha }f\left( x\right) :=\frac{1%
}{\Gamma \left( \alpha \right) }\int_{-\infty }^{x}\frac{f\left( t\right) }{%
\left( x-t\right) ^{1-\alpha }}dt
\end{equation}
and
\begin{equation}\label{A12}
_{-\infty }W_{x}^{\alpha }=_{-\infty }I_{x}^{\alpha }f\left( x\right) :=%
\frac{1}{\Gamma \left( \alpha \right) }\int_{x}^{\infty }\frac{f\left(
t\right) }{\left( t-x\right) ^{1-\alpha }}dt,
\end{equation}
respectively {\rm\cite{RCEC}}.
\end{definition}

\begin{definition} Let a continuous function by parts in $\mathbb{R}=(-\infty,\infty)$. The Liouville fractional integrals $(\mbox{left-sided and right-sided})$ of order $\alpha\in\mathbb{C}$, $\rm Re(\alpha)>0$ of a real function $f$, are defined by
\begin{equation}\label{A13}
I_{+}^{\alpha }f\left( x\right) :=\frac{1}{\Gamma \left( \alpha
\right) }\int_{-\infty }^{x}\frac{f\left( t\right) }{\left( x-t\right)
^{1-\alpha }}dt
\end{equation}
and
\begin{equation}\label{A14}
I_{-}^{\alpha }f\left( x\right) :=\frac{1}{\Gamma \left( \alpha \right) }%
\int_{x}^{\infty }\frac{f\left( t\right) }{\left( t-x\right) ^{1-\alpha }}dt,
\end{equation}
respectively {\rm\cite{AHMJ,IP}}.
\end{definition}

Zoubir \cite{DZO} established the reverse Minkowski's inequality and another result that refers to the inequality via Riemann-Liouville fractional integral according to the following two theorems.

\begin{theorem} Let $\alpha>0$, $p\geq 1$ and $f,g$ two positive functions in $[0, \infty)$, such that $\forall x>0$, $J^{\alpha }f^{p}\left( x\right) <\infty $ and $J^{\alpha }g^{p}\left( x\right) <\infty $. If $0<m\leq \displaystyle\frac{f\left( t\right) }{g\left( t\right) }\leq M$, for $m,M\in\mathbb{R}^{*}_{+}$ and $\forall t\in[0,x]$, then
\begin{equation}\label{A15}
\left( J^{\alpha }f^{p}\left( x\right) \right) ^{\frac{1}{p}}+\left(
J^{\alpha }g^{p}\left( x\right) \right) ^{\frac{1}{p}}\leq c_{1}\left(
J^{\alpha }\left( f+g\right) ^{p}\left( x\right) \right) ^{\frac{1}{p}},
\end{equation}
where $c_{1}=\displaystyle\frac{M\left( m+1\right) +\left( M+1\right) }{\left( m+1\right)\left( M+1\right) }$ {\rm\cite{DZO}}.
\end{theorem}

\begin{theorem} Let $\alpha>0$, $p\geq 1$ and $f,g$ two positive functions in $[0, \infty)$, such that $\forall x>0$, $J^{\alpha }f^{p}\left( x\right) <\infty $ and $J^{\alpha }g^{p}\left( x\right) <\infty $. If $0<m\leq \displaystyle\frac{f\left( t\right) }{g\left( t\right) }\leq M$, for $m,M\in\mathbb{R}^{*}_{+}$ e $\forall t\in[0,x]$, then
\begin{equation}\label{A16}
\left( J^{\alpha }f^{p}\left( x\right) \right) ^{\frac{2}{p}}+\left(
J^{\alpha }g^{p}\left( x\right) \right) ^{\frac{2}{p}}\geq c_{2}\left(
J^{\alpha }f^{p}\left( x\right) \right) ^{\frac{1}{p}}\left( J^{\alpha
}g^{p}\left( x\right) \right) ^{\frac{1}{p}},
\end{equation}
where $c_{2}=\displaystyle\frac{\left( M+1\right) \left( m+1\right) }{M}-2 $ {\rm\cite{DZO}}.
\end{theorem}

In 2014, Chinchane et al. \cite{CHUVP} and Sabrina et al. \cite{SAB} also established the reverse Minkowski's inequality via Hadamard fractional integral as in two theorems below.

\begin{theorem} Let $\alpha>0$, $p\geq 1$ and $f,g$ two positive functions in $[0, \infty)$, such that $\forall x>0$, $H_{1}^{\alpha }f^{p}\left( x\right) <\infty $ and $H_{1}^{\alpha }g^{p}\left( x\right) <\infty $. If $0<m\leq \displaystyle\frac{f\left( t\right) }{g\left( t\right) }\leq M$, for $m,M\in\mathbb{R}^{*}_{+}$ e $\forall t\in[0,x]$, then
\begin{equation}\label{A17}
\left( H_{1}^{\alpha }f^{p}\left( x\right) \right) ^{\frac{1}{p}}+\left(
H_{1}^{\alpha }g^{p}\left( x\right) \right) ^{\frac{1}{p}}\leq c_{1}\left(
H_{1}^{\alpha }\left( f+g\right) ^{p}\left( x\right) \right) ^{\frac{1}{p}},
\end{equation}
where $c_{1}=\displaystyle\frac{M\left( m+1\right) +\left( M+1\right) }{\left( m+1\right) \left( M+1\right) }$ {\rm\cite{CHUVP,SAB}}.
\end{theorem}

\begin{theorem} Let $\alpha>0$, $p\geq 1$ and $f,g$ two positive functions in $[0, \infty)$, such that $\forall x>0$, $H_{1}^{\alpha }f^{p}\left( x\right) <\infty $ and $H_{1}^{\alpha }g^{p}\left( x\right) <\infty $. If $0<m\leq \displaystyle\frac{f\left( t\right) }{g\left( t\right) }\leq M$, for $m,M\in\mathbb{R}^{*}_{+}$ e $\forall t\in[0,x]$, then
\begin{equation}\label{A18}
\left( H_{1}^{\alpha }f^{p}\left( x\right) \right) ^{\frac{2}{p}}+\left(
H_{1}^{\alpha }g^{p}\left( x\right) \right) ^{\frac{2}{p}}\geq c_{2}\left(
H_{1}^{\alpha }f^{p}\left( x\right) \right) ^{\frac{1}{p}}\left(
H_{1}^{\alpha }g^{p}\left( x\right) \right) ^{\frac{1}{p}}
\end{equation}
where $c_{2}= \displaystyle\frac{\left( M+1\right) \left( m+1\right) }{M}-2 $ {\rm\cite{CHUVP,SAB}}.
\end{theorem}

In 2014 Chinchane et al. \cite{CHV4} and recently Chinchane \cite{CHV}, established the reverse Minkowski's inequality via fractional integral of Saigo and the $k$-fractional integral, respectively.

In 2017, Katugampola \cite{UNT} introduced a fractional integral that unifies the six fractional integrals above mentioned. Finally, we introduce this integral and with a theorem we study their respective particular cases.

\begin{definition} Let $\varphi\in X^{p}_{c}(a,b)$, $\alpha>0$ and $\beta,\rho,\eta,\kappa\in\mathbb{R}$. Then, the fractional integrals of a function $f$, left and right, are given by
\begin{equation}\label{A19}
^{\rho }\mathcal{I}_{a+,\eta ,\kappa }^{\alpha ,\beta }\varphi\left( x\right) :=\frac{\rho
^{1-\beta }x^{\kappa }}{\Gamma \left( \alpha \right) }\int_{a}^{x}\frac{%
\tau^{\rho \left( \eta +1\right) -1}}{\left( x^{\rho }-\tau^{\rho }\right)
^{1-\alpha }}\varphi\left( \tau\right) d\tau,\text{ }0\leq a<x<b\leq \infty
\end{equation}
and
\begin{equation}\label{A20}
^{\rho }\mathcal{I}_{b-,\eta, \kappa }^{\alpha ,\beta }\varphi\left( x\right) :=\frac{\rho
^{1-\beta }x^{\rho \eta }}{\Gamma \left( \alpha \right) }\int_{x}^{b}\frac{%
\tau^{\kappa +\rho -1}}{\left( \tau^{\rho }-x^{\rho }\right) ^{1-\alpha }}\varphi\left(
\tau\right) d\tau,\text{ \ }0\leq a<x<b\leq \infty 
\end{equation}
respectively, if integrals exist {\rm\cite{UNT}}. 
\end{definition}

From now on, let's work only with the integral on the left, Eq.(\ref{A19}), because with the right integral we have a similar treatment.

\begin{theorem} Let $\alpha>0$ and $\beta,\rho,\eta,\kappa\in\mathbb{R}$. Then for $f\in X^{p}_{c}(a,b)$, with $a<x<b$, we have {\rm\cite{UNT}}:
\begin{enumerate}

\item \rm \text{For $\kappa=0$, $\eta=0$ and the limit $\rho\rightarrow 1$}, at \textnormal{Eq.(\ref{A19})}, we get the Riemann-Liouville fractional integral, i.e; Eq.(\ref{A3}).

\item With $\beta=\alpha$, $\kappa=0$, $\eta=0$, we take the limit $\rho\rightarrow 0^{+}$ and using the $\ell$'Hospital role, at \textnormal{Eq.(\ref{A19})}, we get the Hadamard fractional integral, i.e; Eq.(\ref{A5}).

\item In the case $\beta=0$ and $\kappa=-\rho(\alpha+\eta)$, at \textnormal{Eq.(\ref{A19})}, we get the Erdélyi-Kober fractional integral, i.e; Eq.(\ref{A7}).

\item For $\beta=\alpha$, $\kappa=0$ and $\eta=0$, at \textnormal{Eq.(\ref{A19})}, we get Katugampola fractional integral, i.e; Eq.(\ref{A9}).

\item With $\kappa=0$, $\eta=0$, $a=-\infty$ and take the limit $\rho\rightarrow 1$, at \textnormal{Eq.(\ref{A19})}, we get Weyl fractional integral, i.e; Eq.(\ref{A11}).

\item With $\kappa=0$, $\eta=0$, $a=0$ and take the limit $\rho\rightarrow 1$, at \textnormal{Eq.(\ref{A19})}, we get Liouville fractional integral, i.e; Eq.(\ref{A13}).
\end{enumerate}
\end{theorem}


\section{Reverse Minkowski fractional integral inequality}

In this section, our main contribution, we establish and prove the reverse Minkowski's inequality via generalized fractional integral Eq.(\ref{A19}) and a theorem that refers to the reverse Minkowski's inequality.

\begin{theorem} Let $\alpha>0$, $\rho,\eta,\kappa,\beta\in\mathbb{R}$ and $p\geq 1$. Let $f,g\in X^{p}_{c}(a,x)$ two positive functions in $[0,\infty)$, such that $\forall x>a$, $^{\rho }\mathcal{I}_{a+,\eta ,\kappa }^{\alpha ,\beta }f^{p}\left( x\right) <\infty $ and $^{\rho }\mathcal{I}_{a+,\eta ,\kappa }^{\alpha ,\beta }g^{p}\left( x\right) <\infty $. If $0<m\leq \displaystyle\frac{f\left( t\right) }{g\left( t\right) }\leq M$, for $m,M\in\mathbb{R}^{*}_{+}$ and $\forall t\in \left[ a,x\right] $, then 
\begin{equation}\label{A21}
\left( ^{\rho }\mathcal{I}_{a+,\eta ,\kappa }^{\alpha ,\beta }f^{p}\left( x\right)
\right) ^{\frac{1}{p}}+\left( ^{\rho }\mathcal{I}_{a+,\eta ,\kappa }^{\alpha ,\beta
}g^{p}\left( x\right) \right) ^{\frac{1}{p}}\leq c_{1}\left( ^{\rho
}\mathcal{I}_{a+,\eta ,\kappa }^{\alpha ,\beta }\left( f+g\right) ^{p}\left( x\right)
\right) ^{\frac{1}{p}},
\end{equation}
with $c_{1}=\displaystyle\frac{M\left( m+1\right) +\left( M+1\right) }{\left( m+1\right) \left( M+1\right) }$.
\end{theorem}

\begin{proof} Using the condition $\displaystyle\frac{f\left( t\right) }{g\left( t\right) }\leq M$, $t\in[a,x]$, we can write
\begin{equation*}
f\left( t\right) \leq M\left( f\left( t\right) +g\left( t\right) \right)
-Mf\left( t\right) ,
\end{equation*}
which implies,
\begin{equation}\label{A22}
\left( M+1\right) ^{p}f^{p}\left( t\right) \leq M^{p}\left( f\left( t\right)
+g\left( t\right) \right) ^{p}.
\end{equation}

Multiplying by $\displaystyle\frac{\rho ^{1-\beta }x^{\kappa }t^{\rho \left( \eta +1\right) -1}}{\Gamma\left( \alpha \right) \left( x^{\rho }-t^{\rho }\right) ^{1-\alpha }}$ 
both sides of {\rm Eq.(\ref{A22})} and integrating with respect to the variable $t$, we have
\begin{equation}\label{A23}
\frac{\left( M+1\right) ^{p}\rho ^{1-\beta }x^{\kappa }}{\Gamma \left(
\alpha \right) }\int_{a}^{x}\frac{t^{\rho \left( \eta +1\right) -1}}{\left(
x^{\rho }-t^{\rho }\right) ^{1-\alpha }}f^{p}\left( t\right) dt\leq \frac{%
M^{p}\rho ^{1-\beta }x^{\kappa }}{\Gamma \left( \alpha \right) }\int_{a}^{x}%
\frac{t^{\rho \left( \eta +1\right) -1}}{\left( x^{\rho }-t^{\rho }\right)
^{1-\alpha }}\left( f+g\right) ^{p}\left( t\right) dt.
\end{equation}

Consequently, we can write
\begin{equation}\label{A23}
\left( ^{\rho }\mathcal{I}_{a+,\eta ,\kappa }^{\alpha ,\beta }f^{p}\left( x\right)
\right) ^{\frac{1}{p}}\leq \frac{M}{M+1}\left( ^{\rho }\mathcal{I}_{a+,\eta ,\kappa
}^{\alpha ,\beta }\left( f+g\right) ^{p}\left( x\right) \right) ^{\frac{1}{p}}.
\end{equation}

On the other hand, as $mg\left( t\right) \leq f\left( t\right) $, follow
\begin{equation}\label{A24}
\left( 1+\frac{1}{m}\right) ^{p}g^{p}\left( t\right) \leq \left( \frac{1}{m%
}\right) ^{p}\left( f\left( t\right) +g\left( t\right) \right) ^{p}.
\end{equation}

Further, multiplying by $\displaystyle\frac{\rho ^{1-\beta }x^{\kappa }t^{\rho \left( \eta +1\right) -1}}{\Gamma\left( \alpha \right) \left( x^{\rho }-t^{\rho }\right) ^{1-\alpha }}$ both sides of {\rm Eq.(\ref{A24})} and integrating with respect to the variable $t$, we have
\begin{equation}\label{A25}
\left( ^{\rho }\mathcal{I}_{a+,\eta ,\kappa }^{\alpha ,\beta }g^{p}\left( t\right)
\right) ^{\frac{1}{p}}\leq \frac{1}{m+1}\left( ^{\rho }\mathcal{I}_{a+,\eta ,\kappa
}^{\alpha ,\beta }\left( f+g\right) ^{p}\left( t\right) \right) ^{\frac{1}{p}}.
\end{equation}

From {\rm Eq.(\ref{A23})} and {\rm Eq.(\ref{A25})}, the result follows.
\end{proof}

Eq.(\ref{A21}) is the so-called reverse Minkowski's inequality associated with the Katugampola fractional integral.

\begin{theorem} Let $\alpha>0$, $\rho,\eta,\kappa,\beta\in\mathbb{R}$ and $p\geq 1$. Let $f,g\in X^{p}_{c}(a,x)$ be two positive functions in $[0,\infty)$, such that $\forall x>a$, $^{\rho }\mathcal{I}_{a+,\eta ,\kappa }^{\alpha ,\beta }f^{p}\left( x\right) <\infty $ and $^{\rho }\mathcal{I}_{a+,\eta ,\kappa }^{\alpha ,\beta }g^{p}\left( x\right) <\infty $. If $0<m\leq \displaystyle\frac{f\left( t\right) }{g\left( t\right) }\leq M$, for $m,M\in\mathbb{R}^{*}_{+}$ and $\forall t\in \left[ a,x\right] $, then
\begin{equation}\label{A26}
\left( ^{\rho }\mathcal{I}_{a+,\eta ,\kappa }^{\alpha ,\beta }f^{p}\left( x\right)
\right) ^{\frac{2}{p}}+\left( ^{\rho }\mathcal{I}_{a+,\eta ,\kappa }^{\alpha ,\beta
}g^{p}\left( x\right) \right) ^{\frac{2}{p}}\geq c_{2}\left( ^{\rho
}\mathcal{I}_{a+,\eta ,\kappa }^{\alpha ,\beta }f^{p}\left( x\right) \right) ^{\frac{1%
}{p}}\left( ^{\rho }\mathcal{I}_{a+,\eta ,\kappa }^{\alpha ,\beta }g^{p}\left(
x\right) \right) ^{\frac{1}{p}}
\end{equation}
with $c_{2}=\displaystyle\frac{\left( M+1\right) \left( m+1\right) }{M}-2$.
\end{theorem}

\begin{proof}
 Carrying out the product between {\rm Eq.(\ref{A23})} and {\rm Eq.(\ref{A25})}, we have
\begin{equation}\label{A27}
\frac{\left( M+1\right) \left( m+1\right) }{M}\left( ^{\rho }\mathcal{I}_{a+,\eta
,\kappa }^{\alpha ,\beta }f^{p}\left( x\right) \right) ^{\frac{1}{p}}\left(
^{\rho }\mathcal{I}_{a+,\eta ,\kappa }^{\alpha ,\beta }g^{p}\left( x\right) \right) ^{%
\frac{1}{p}}\leq \left( ^{\rho }\mathcal{I}_{a+,\eta ,\kappa }^{\alpha ,\beta }\left(
f+g\right) ^{p}\left( x\right) \right) ^{\frac{2}{p}}.
\end{equation}

Using the Minkowski's inequality, on the right side of {\rm Eq.(\ref{A27})}, we have
\begin{eqnarray}
&&\frac{\left( M+1\right) \left( m+1\right) }{M}\left( ^{\rho }\mathcal{I}_{a+,\eta
,\kappa }^{\alpha ,\beta }f^{p}\left( x\right) \right) ^{\frac{1}{p}}\left(
^{\rho }\mathcal{I}_{a+,\eta ,\kappa }^{\alpha ,\beta }g^{p}\left( x\right) \right) ^{%
\frac{1}{p}}  \label{A28} \\
&\leq &\left( \left( ^{\rho }\mathcal{I}_{a+,\eta ,\kappa }^{\alpha ,\beta
}f^{p}\left( x\right) \right) ^{\frac{1}{p}}+\left( ^{\rho }\mathcal{I}_{a+,\eta
,\kappa }^{\alpha ,\beta }g^{p}\left( x\right) \right) ^{\frac{1}{p}}\right)
^{2}.  \notag
\end{eqnarray}
So, from {\rm Eq.(\ref{A28})}, we conclude that
\begin{eqnarray*}
&&\left( \frac{\left( M+1\right) \left( m+1\right) }{M}-2\right) \left(
^{\rho }\mathcal{I}_{a+,\eta ,\kappa }^{\alpha ,\beta }f^{p}\left( x\right) \right) ^{%
\frac{1}{p}}\left( ^{\rho }\mathcal{I}_{a+,\eta ,\kappa }^{\alpha ,\beta }g^{p}\left(
x\right) \right) ^{\frac{1}{p}} \\
&\leq &\left( ^{\rho }\mathcal{I}_{a+,\eta ,\kappa }^{\alpha ,\beta }f^{p}\left(
x\right) \right) ^{\frac{2}{p}}+\left( ^{\rho }\mathcal{I}_{a+,\eta ,\kappa }^{\alpha
,\beta }g^{p}\left( x\right) \right) ^{\frac{2}{p}}.
\end{eqnarray*}
\end{proof}

Note that, if $\beta =\alpha$, $\kappa = 0$, $\eta=0$ and the limit $\rho\rightarrow 1$, in Eq.(\ref{A19}), we recover Riemann-Liouville fractional integral, Eq.(\ref{A3}). In this sense, choosing $a+=0 $, and substituting in Theorem 8 and Theorem 9, we obtain, as particular cases, the respective Theorem 3 and Theorem 4, which correspond to the inequality via Riemann-Liouville fractional integral. On the other hand, if $\beta=\alpha$, $\kappa = 0 $, $\eta = 0$, and the limit $\rho \rightarrow 0+$ and using the $\ell$'Hospital rule, in Eq.(\ref{A19}), we obtain the Hadamard fractional integral, Eq.(\ref{A5}). Similarly, choosing $a=1$ and substituting in Theorem 8 and Theorem 9, we obtain, as particular cases, the Theorem 5 and Theorem 6, respectively.

\section{Other fractional integral inequalities}

In this section we generalize the results discussed by Chinchane \cite{CHV}, Sulaiman \cite{SULA} and Sroysang \cite{PALM} on the reverse Minkowski's inequality via Riemann integral, using the fractional integral proposed by Katugampola \cite{UNT}.

\begin{theorem} Let $\alpha>0$, $\rho,\eta,\kappa,\beta\in\mathbb{R}$, $p\geq 1$ and $\frac{1}{p}+\frac{1}{q}=1$. Let $f,g\in X^{p}_{c}(a,x)$ be two positive functions in $[0,\infty)$, such that $\forall x>a$, $^{\rho }\mathcal{I}_{a+,\eta ,\kappa }^{\alpha ,\beta }f\left( x\right) <\infty $ and $^{\rho }\mathcal{I}_{a+,\eta ,\kappa }^{\alpha ,\beta }g\left( x\right) <\infty $. If $0<m\leq \displaystyle\frac{f\left( t\right) }{g\left( t\right) }\leq M$, for $m,M\in\mathbb{R}^{*}_{+}$ and $\forall t\in \left[ a,x\right] $, then 
\begin{equation}\label{A29}
\left( ^{\rho }\mathcal{I}_{a+,\eta ,\kappa }^{\alpha ,\beta }f\left( x\right) \right)
^{\frac{1}{p}}\left( ^{\rho }\mathcal{I}_{a+,\eta ,\kappa }^{\alpha ,\beta }g\left(
x\right) \right) ^{\frac{1}{q}}\leq \left( \frac{M}{m}\right) ^{\frac{1}{pq}
}\left( ^{\rho }\mathcal{I}_{a+,\eta ,\kappa }^{\alpha ,\beta }f^{\frac{1}{p}}\left(
x\right) g^{\frac{1}{q}}\left( x\right) \right) .
\end{equation}
\end{theorem}
\begin{proof}

Using the condition $\displaystyle\frac{f\left( t\right) }{g\left( t\right) }\leq M$, $t\in [a,x]$ with $x>a$, we have
\begin{equation}\label{A30}
f\left( t\right) \leq Mg\left( t\right) \Rightarrow g^{\frac{1}{q}}\left(
t\right) \geq M^{-\frac{1}{q}}f^{\frac{1}{q}}\left( t\right) .
\end{equation}

Multiplying by $f^{\frac{1}{p}}\left( t\right) $ both sides of {\rm Eq.(\ref{A30})}, we can rewrite it as follows
\begin{equation}\label{A31}
f^{\frac{1}{p}}\left( t\right) g^{\frac{1}{q}}\left( t\right) \geq M^{-\frac{%
1}{q}}f\left( t\right) .
\end{equation}

Now, multiplying by $\displaystyle\frac{\rho ^{1-\beta }x^{\kappa }t^{\rho \left( \eta +1\right) -1}}{\Gamma \left( \alpha \right) \left( x^{\rho }-t^{\rho }\right) ^{1-\alpha }}$ both sides of {\rm Eq.(\ref{A31})} and integrating with respect to the variable $t$, we have
\begin{equation}\label{A32}
\int_{a}^{x}\frac{\rho ^{1-\beta }x^{\kappa }t^{\rho \left( \eta +1\right)
-1}}{\Gamma \left( \alpha \right) \left( x^{\rho }-t^{\rho }\right)
^{1-\alpha }}M^{-\frac{1}{q}}f\left( t\right) dt\leq \int_{a}^{x}\frac{\rho
^{1-\beta }x^{\kappa }t^{\rho \left( \eta +1\right) -1}}{\Gamma \left(
\alpha \right) \left( x^{\rho }-t^{\rho }\right) ^{1-\alpha }}f^{\frac{1}{p}%
}\left( t\right) g^{\frac{1}{q}}\left( t\right) dt.
\end{equation}

So, the inequality follows
\begin{equation}\label{A33}
M^{-\frac{1}{pq}}\left( ^{\rho }\mathcal{I}_{a+,\eta ,\kappa }^{\alpha ,\beta }f\left(
x\right) \right) ^{\frac{1}{p}}\leq \left( ^{\rho }\mathcal{I}_{a+,\eta ,\kappa
}^{\alpha ,\beta }f^{\frac{1}{p}}\left( x\right) g^{\frac{1}{q}}\left(
x\right) \right) ^{\frac{1}{p}}.
\end{equation}

On the order hand, we have
\begin{equation}\label{A34}
m^{\frac{1}{p}}g^{\frac{1}{p}}\left( t\right) \leq f^{\frac{1}{p}}\left(t\right), \text{ } x>a.
\end{equation}

Multiplying by $g^{\frac{1}{q}}(t)$ both sides of {\rm Eq.(\ref{A34})} and using the relation $\frac{1}{p}+\frac{1}{q}=1$, we have
\begin{equation}\label{A35}
m^{\frac{1}{p}}g\left( t\right) \leq f^{\frac{1}{p}}\left( t\right) g^{\frac{%
1}{q}}\left( t\right) .
\end{equation}

Multiplying by $\displaystyle\frac{\rho ^{1-\beta }x^{\kappa }t^{\rho \left( \eta +1\right) -1}}{\Gamma\left( \alpha \right) \left( x^{\rho }-t^{\rho }\right) ^{1-\alpha }}$ both sides of {\rm Eq.(\ref{A35})} and integrating with respect to the variable $t$, we have
\begin{equation}\label{A36}
m^{\frac{1}{pq}}\left( ^{\rho }\mathcal{I}_{a+,\eta ,\kappa }^{\alpha ,\beta }g\left(
x\right) \right) ^{\frac{1}{q}}\leq \left( ^{\rho }\mathcal{I}_{a+,\eta ,\kappa
}^{\alpha ,\beta }f^{\frac{1}{p}}\left( x\right) g^{\frac{1}{q}}\left(
x\right) \right) ^{\frac{1}{q}}.
\end{equation}

Evaluating the product between {\rm Eq.(\ref{A33})} and {\rm Eq.(\ref{A36})} and using the relation $\frac{1}{p}+\frac{1}{q}=1$, we conclude that
\begin{equation*}
\left( ^{\rho }\mathcal{I}_{a+,\eta ,\kappa }^{\alpha ,\beta }f\left( x\right) \right)
^{\frac{1}{p}}\left( ^{\rho }\mathcal{I}_{a+,\eta ,\kappa }^{\alpha ,\beta }g\left(
x\right) \right) ^{\frac{1}{q}}\leq \left( \frac{M}{m}\right) ^{\frac{1}{pq}%
}\left( ^{\rho }\mathcal{I}_{a+,\eta ,\kappa }^{\alpha ,\beta }f^{\frac{1}{p}}\left(
x\right) g^{\frac{1}{q}}\left( x\right) \right) ^{\frac{1}{p}}.
\end{equation*}
\end{proof}

\begin{theorem} Let $\alpha>0$, $\rho,\eta,\kappa,\beta\in\mathbb{R}$, $p\geq 1$ and $\frac{1}{p}+\frac{1}{q}=1$. Let $f,g\in X^{p}_{c}(a,x)$ be two positive functions in $[0,\infty)$, such that $\forall x>a$, $^{\rho }\mathcal{I}_{a+,\eta ,\kappa }^{\alpha ,\beta }f^{p}\left( x\right) <\infty $, $^{\rho }\mathcal{I}_{a+,\eta ,\kappa }^{\alpha ,\beta }f^{q}\left( x\right) <\infty $, $^{\rho }\mathcal{I}_{a+,\eta ,\kappa }^{\alpha ,\beta }g^{p}\left( x\right) <\infty $ and $^{\rho }\mathcal{I}_{a+,\eta ,\kappa }^{\alpha ,\beta }g^{q}\left( x\right) <\infty $. If $0<m\leq \displaystyle\frac{f\left( t\right) }{g\left( t\right) }\leq M$, for $m,M\in\mathbb{R}^{*}_{+}$ and $\forall t\in \left[ a,x\right] $, then  
\begin{equation}\label{A29}
^{\rho }\mathcal{I}_{a+,\eta ,\kappa }^{\alpha ,\beta }f\left( x\right) g\left(
x\right) \leq c_{3}\left( ^{\rho }\mathcal{I}_{a+,\eta ,\kappa }^{\alpha ,\beta
}\left( f^{p}+g^{p}\right) \left( x\right) \right) +c_{4}\left( ^{\rho
}\mathcal{I}_{a+,\eta ,\kappa }^{\alpha ,\beta }\left( f^{q}+g^{q}\right) \left(
x\right) \right) ,
\end{equation}
with $c_{3}=\displaystyle\frac{2^{p-1}M^{p}}{p\left( M+1\right) ^{p}}$ and $c_{4}=\displaystyle\frac{2^{p-1}}{q\left( m+1\right) ^{q}}$.
\end{theorem}
\begin{proof} Using the hypothesis, we have the following identity
\begin{equation}\label{B1}
\left( M+1\right) ^{p}f^{p}\left( t\right) \leq M^{p}\left( f+g\right)
^{p}\left( t\right) .
\end{equation}

Multiplying by $\displaystyle\frac{\rho ^{1-\beta }x^{\kappa }t^{\rho \left( \eta +1\right) -1}}{\Gamma\left( \alpha \right) \left( x^{\rho }-t^{\rho }\right) ^{1-\alpha }}$ both sides of {\rm Eq.(\ref{B1})} and integrating with respect to the variable $t$, we get
\begin{equation*}
\int_{a}^{x}\frac{\rho ^{1-\beta }x^{\kappa }t^{\rho \left( \eta +1\right)
-1}}{\Gamma \left( \alpha \right) \left( x^{\rho }-t^{\rho }\right)
^{1-\alpha }}\left( M+1\right) ^{p}f^{p}\left( t\right) dt\leq \int_{a}^{x}%
\frac{\rho ^{1-\beta }x^{\kappa }t^{\rho \left( \eta +1\right) -1}}{\Gamma
\left( \alpha \right) \left( x^{\rho }-t^{\rho }\right) ^{1-\alpha }}%
M^{p}\left( f+g\right) ^{p}\left( t\right) dt.
\end{equation*}

In this way, we have
\begin{equation}\label{B2}
^{\rho }\mathcal{I}_{a+,\eta ,\kappa }^{\alpha ,\beta }f^{p}\left( x\right) \leq \frac{%
M^{p}}{\left( M+1\right) ^{p}}\; ^{\rho }\mathcal{I}_{a+,\eta ,\kappa }^{\alpha ,\beta
}\left( f+g\right) ^{p}\left( x\right) .
\end{equation}

On the other hand, as $0<m<\frac{f\left( t\right) }{g\left( t\right) }$, $t\in \left( a,x\right) $, we have
\begin{equation}\label{B3}
\left( m+1\right) ^{q}g^{q}\left( t\right) \leq \left( f+g\right) ^{q}\left(t\right) .
\end{equation}

Again, multiplying by $\displaystyle\frac{\rho ^{1-\beta }x^{\kappa }t^{\rho \left( \eta +1\right) -1}}{\Gamma \left( \alpha \right) \left( x^{\rho }-t^{\rho }\right) ^{1-\alpha }}$ both sides of {\rm Eq.(\ref{B3})} and integrating with respect to the variable $t$, we get
\begin{equation}\label{B4}
^{\rho }\mathcal{I}_{a+,\eta ,\kappa }^{\alpha ,\beta }g^{q}\left( x\right) \leq \frac{%
1}{\left( m+1\right) ^{q}} \;^{\rho }\mathcal{I}_{a+,\eta ,\kappa }^{\alpha ,\beta
}\left( f+g\right) ^{q}\left( x\right).
\end{equation}

Considering Young's inequality, {\rm\cite{KREZ}}
\begin{equation}\label{B5}
f\left( t\right) g\left( t\right) \leq \frac{f^{p}\left( t\right) }{p}+\frac{%
g^{q}\left( t\right) }{q},
\end{equation}
multiplying by $\displaystyle\frac{\rho ^{1-\beta }x^{\kappa }t^{\rho \left( \eta +1\right) -1}}{\Gamma \left( \alpha \right) \left( x^{\rho }-t^{\rho }\right) ^{1-\alpha }}$ both sides of {\rm Eq.(\ref{B5})} and integrating with respect to the variable $t$, we have
\begin{equation}\label{B6}
^{\rho }\mathcal{I}_{a+,\eta ,\kappa }^{\alpha ,\beta }\left( fg\right) \left(
x\right) \leq \frac{1}{p}\left( ^{\rho }\mathcal{I}_{a+,\eta ,\kappa }^{\alpha ,\beta
}f^{p}\left( x\right) \right) +\frac{1}{q}\left( ^{\rho }\mathcal{I}_{a+,\eta ,\kappa
}^{\alpha ,\beta }g^{q}\left( x\right) \right) .
\end{equation}

Thus, using {\rm Eq.(\ref{B2})}, {\rm Eq.(\ref{B4})} and {\rm Eq.(\ref{B6})}, we get
\begin{eqnarray}
^{\rho }\mathcal{I}_{a+,\eta ,\kappa }^{\alpha ,\beta }\left( fg\right) \left(
x\right)  &\leq &\frac{1}{p}\left( ^{\rho }\mathcal{I}_{a+,\eta ,\kappa }^{\alpha
,\beta }f^{p}\left( x\right) \right) +\frac{1}{q}\left( ^{\rho }\mathcal{I}_{a+,\eta
,\kappa }^{\alpha ,\beta }g^{q}\left( x\right) \right)   \notag  \label{B7}
\\
&\leq &\left( ^{\rho }\mathcal{I}_{a+,\eta ,\kappa }^{\alpha ,\beta }f^{p}\left(
x\right) \right) +\left( ^{\rho }\mathcal{I}_{a+,\eta ,\kappa }^{\alpha ,\beta
}g^{q}\left( x\right) \right)   \notag \\
&\leq &\frac{M^{p}}{p\left( M+1\right) ^{p}}\left( ^{\rho }\mathcal{I}_{a+,\eta
,\kappa }^{\alpha ,\beta }\left( f+g\right) ^{p}\left( x\right) \right)  
\notag \\
&&+\frac{1}{q\left( m+1\right) ^{q}}\left( ^{\rho }\mathcal{I}_{a+,\eta ,\kappa
}^{\alpha ,\beta }\left( f+g\right) ^{q}\left( x\right) \right) .
\end{eqnarray}

Using the following inequality, $\left( a+b\right) ^{r}\leq 2^{p-1}\left( a^{r}+b^{r}\right)$,  $r>1$, $a,b\geq 0$, we get
\begin{equation}\label{B8}
^{\rho }\mathcal{I}_{a+,\eta ,\kappa }^{\alpha ,\beta }\left( f+g\right) ^{p}\left(
x\right) \leq 2^{p-1\rho }\mathcal{I}_{a+,\eta ,\kappa }^{\alpha ,\beta }\left(
f^{p}+g^{p}\right) \left( x\right) 
\end{equation}
and
\begin{equation}\label{B9}
^{\rho }\mathcal{I}{a+,\eta ,\kappa }^{\alpha ,\beta }\left( f+g\right) ^{q}\left(
x\right) \leq 2^{q-1\rho }\mathcal{I}_{a+,\eta ,\kappa }^{\alpha ,\beta }\left(
f^{q}+g^{q}\right) \left( x\right) .
\end{equation}

Thus, replacing {\rm Eq.(\ref{B8})} and {\rm Eq.(\ref{B9})} at {\rm Eq.(\ref{B7})}, we conclude that
\begin{equation*}
^{\rho }\mathcal{I}_{a+,\eta ,\kappa }^{\alpha ,\beta }\left( fg\right) \left(
x\right) \leq \frac{2^{p-1}M^{p}}{p\left( M+1\right) ^{p}}\left( ^{\rho
}\mathcal{I}_{a+,\eta ,\kappa }^{\alpha ,\beta }\left( f^{p}+g^{p}\right) \left(
x\right) \right) +\frac{2^{q-1}}{q\left( m+1\right) ^{q}}\left( ^{\rho
}\mathcal{I}_{a+,\eta ,\kappa }^{\alpha ,\beta }\left( f^{q}+g^{q}\right) \left(
x\right) \right) .
\end{equation*}
\end{proof}

\begin{theorem} Let $\alpha>0$, $\rho,\eta,\kappa,\beta\in\mathbb{R}$ and $p\geq 1$. Let $f,g\in X^{p}_{c}(a,x)$ be two positive functions in $[0,\infty)$, such that $\forall x>a$, $^{\rho }\mathcal{I}_{a+,\eta ,\kappa }^{\alpha ,\beta }f^{p}\left( x\right) <\infty $ and $^{\rho }\mathcal{I}_{a+,\eta ,\kappa }^{\alpha ,\beta }g^{p}\left( x\right) <\infty $. If $0<m\leq \displaystyle\frac{f\left( t\right) }{g\left( t\right) }\leq M$, for $m,M\in\mathbb{R}^{*}_{+}$ and $\forall t\in \left[ a,x\right] $, then
\begin{eqnarray}
\frac{M+1}{M-c}\left( ^{\rho }\mathcal{I}_{a+,\eta ,\kappa }^{\alpha ,\beta }\left(
f\left( x\right) -cg\left( x\right) \right) \right) ^{\frac{1}{p}} &\leq
&\left( ^{\rho }\mathcal{I}_{a+,\eta ,\kappa }^{\alpha ,\beta }f^{p}\left( x\right)
\right) ^{\frac{1}{p}}+\left( ^{\rho }\mathcal{I}_{a+,\eta ,\kappa }^{\alpha ,\beta
}g^{p}\left( x\right) \right) ^{\frac{1}{p}}  \notag  \label{A29} \\
&\leq &\frac{m+1}{m-c}\left( ^{\rho }\mathcal{I}_{a+,\eta ,\kappa }^{\alpha ,\beta
}\left( f\left( x\right) -cg\left( x\right) \right) \right) ^{\frac{1}{p}}
\end{eqnarray}
\end{theorem}
\begin{proof}
By hypothesis $0<c<m\leq M$, so
\begin{equation*}
mc\leq Mc\Rightarrow mc+m\leq mc+M\leq Mc+M\Rightarrow \left( M+1\right)
\left( m-c\right) \leq \left( m+1\right) \left( M-c\right) .
\end{equation*}

Thus, we conclude that
\begin{equation*}
\frac{M+1}{M-c}\leq \frac{m+1}{m-c}.
\end{equation*}

Also, we have
\begin{equation*}
m-c\leq \frac{f\left( t\right) -cg\left( t\right) }{g\left( t\right) }\leq M-c
\end{equation*}
which implies,
\begin{equation}\label{D1}
\frac{\left( f\left( t\right) -cg\left( t\right) \right) ^{p}}{\left(
M-c\right) ^{p}}\leq g^{p}\left( t\right) \leq \frac{\left( f\left( t\right)
-cg\left( t\right) \right) ^{p}}{\left( m-c\right) ^{p}}.
\end{equation}

Again, we have
\begin{equation*}
\frac{1}{M}\leq \frac{g\left( t\right) }{f\left( t\right) }\leq \frac{1}{m}%
\Rightarrow \frac{m-c}{cm}\leq \frac{f\left( t\right) -cg\left( t\right) }{%
cf\left( t\right) }\leq \frac{M-c}{cM},
\end{equation*}
which implies,
\begin{equation}\label{D2}
\left( \frac{M}{M-c}\right) ^{p}\left( f\left( t\right) -cg\left( t\right)
\right) ^{p}\leq f^{p}\left( t\right) \leq \left( \frac{m}{m-c}\right)
^{p}\left( f\left( t\right) -cg\left( t\right) \right) ^{p}.
\end{equation}

Multiplying by $\displaystyle\frac{\rho ^{1-\beta }x^{\kappa }t^{\rho \left( \eta +1\right) -1}}{\Gamma\left( \alpha \right) \left( x^{\rho }-t^{\rho }\right) ^{1-\alpha }}$ both sides of {\rm Eq.(\ref{D1})} and integrating with respect to the variable $t$, we have
\begin{eqnarray*}
\int_{a}^{x}\frac{\rho ^{1-\beta }x^{\kappa }t^{\rho \left( \eta +1\right)
-1}}{\Gamma \left( \alpha \right) \left( x^{\rho }-t^{\rho }\right)
^{1-\alpha }}\frac{\left( f\left( t\right) -cg\left( t\right) \right) ^{p}}{%
\left( M-c\right) ^{p}}dt &\leq &\int_{a}^{x}\frac{\rho ^{1-\beta }x^{\kappa
}t^{\rho \left( \eta +1\right) -1}}{\Gamma \left( \alpha \right) \left(
x^{\rho }-t^{\rho }\right) ^{1-\alpha }}g^{p}\left( t\right) dt \\
&\leq &\int_{a}^{x}\frac{\rho ^{1-\beta }x^{\kappa }t^{\rho \left( \eta
+1\right) -1}}{\Gamma \left( \alpha \right) \left( x^{\rho }-t^{\rho
}\right) ^{1-\alpha }}\frac{\left( f\left( t\right) -cg\left( t\right)
\right) ^{p}}{\left( m-c\right) ^{p}}dt.
\end{eqnarray*}

In this way, we obtain
\begin{eqnarray}\label{D3}
\frac{1}{M-c}\left( ^{\rho }\mathcal{I}_{a+,\eta ,\kappa }^{\alpha ,\beta }\left(
f\left( x\right) -cg\left( x\right) \right) ^{p}\right) ^{\frac{1}{p}} &\leq
&\left( ^{\rho }\mathcal{I}_{a+,\eta ,\kappa }^{\alpha ,\beta }g^{p}\left( x\right)
\right) ^{\frac{1}{p}} \\
&\leq &\frac{1}{m-c}\left( ^{\rho }\mathcal{I}_{a+,\eta ,\kappa }^{\alpha ,\beta
}\left( f\left( x\right) -cg\left( x\right) \right) ^{p}\right) ^{\frac{1}{p}}.\notag 
\end{eqnarray}

Realizing the same procedure as in {\rm Eq.(\ref{D2})}, we have
\begin{eqnarray}\label{D4}
\frac{M}{M-c}\left( ^{\rho }\mathcal{I}_{a+,\eta ,\kappa }^{\alpha ,\beta }\left(
f\left( x\right) -cg\left( x\right) \right) ^{p}\right) ^{\frac{1}{p}} &\leq
&\left( ^{\rho }\mathcal{I}_{a+,\eta ,\kappa }^{\alpha ,\beta }f^{p}\left( x\right)
\right) ^{\frac{1}{p}} \\
&\leq &\frac{m}{m-c}\left( ^{\rho }\mathcal{I}_{a+,\eta ,\kappa }^{\alpha ,\beta
}\left( f\left( x\right) -cg\left( x\right) \right) ^{p}\right) ^{\frac{1}{p}}. \notag 
\end{eqnarray}

Adding {\rm Eq.(\ref{D3})} and {\rm Eq.(\ref{D4})}, we conclude that
\begin{eqnarray*}
\frac{M+1}{M-c}\left( ^{\rho }\mathcal{I}_{a+,\eta ,\kappa }^{\alpha ,\beta }\left(
f\left( x\right) -cg\left( x\right) \right) ^{p}\right) ^{\frac{1}{p}} &\leq
&\left( ^{\rho }\mathcal{I}_{a+,\eta ,\kappa }^{\alpha ,\beta }f^{p}\left( x\right)
\right) ^{\frac{1}{p}}+\left( ^{\rho }\mathcal{I}_{a+,\eta ,\kappa }^{\alpha ,\beta
}g^{p}\left( x\right) \right) ^{\frac{1}{p}} \\
&\leq &\frac{m+1}{m-c}\left( ^{\rho }\mathcal{I}_{a+,\eta ,\kappa }^{\alpha ,\beta
}\left( f\left( x\right) -cg\left( x\right) \right) ^{p}\right) ^{\frac{1}{p}}.
\end{eqnarray*}

\end{proof}

\begin{theorem} Let $\alpha>0$, $\rho,\eta,\kappa,\beta\in\mathbb{R}$ and $p\geq 1$. Let $f,g\in X^{p}_{c}(a,x)$ be two positive functions in $[0,\infty)$, such that $\forall x>a$, $^{\rho }\mathcal{I}_{a+,\eta ,\kappa }^{\alpha ,\beta }f^{p}\left( x\right) <\infty $ and $^{\rho }\mathcal{I}_{a+,\eta ,\kappa }^{\alpha ,\beta }g^{p}\left( x\right) <\infty $. If $0 \leq a \leq f(t)\leq A$ and $0\leq b \leq g(t)\leq B$, $\forall t\in \left[ a,x\right] $, then
\begin{equation}
\left( ^{\rho }\mathcal{I}_{a+,\eta ,\kappa }^{\alpha ,\beta }f^{p}\left( x\right)
\right) ^{\frac{1}{p}}+\left( ^{\rho }\mathcal{I}_{a+,\eta ,\kappa }^{\alpha ,\beta
}g^{p}\left( x\right) \right) ^{\frac{1}{p}}\leq c_{5}\left( ^{\rho
}\mathcal{I}_{a+,\eta ,\kappa }^{\alpha ,\beta }\left( f+g\right) ^{p}\left( x\right)
\right) ^{\frac{1}{p}},
\end{equation}
with $c_{5}=\displaystyle\frac{A\left( a+B\right) +B\left( A+b\right) }{\left( A+b\right)\left( a+B\right) }$.
\end{theorem}
\begin{proof}

By hypothesis, it follows that
\begin{equation}\label{C1}
\frac{1}{B}\leq \frac{1}{g\left( t\right) }\leq \frac{1}{b}.
\end{equation}

Realizing the product between {\rm Eq.(\ref{C1})} and $0<a\leq f\left( t\right) \leq A$, we have
\begin{equation}\label{C2}
\frac{a}{B}\leq \frac{f\left( t\right) }{g\left( t\right) }\leq \frac{A}{b}.
\end{equation}

From {\rm Eq.(\ref{C2})}, we get
\begin{equation}\label{C3}
g^{p}\left( t\right) \leq \left( \frac{B}{a+B}\right) ^{p}\left( f\left(
t\right) +g\left( t\right) \right) ^{p}
\end{equation}
and
\begin{equation}\label{C4}
f^{p}\left( t\right) \leq \left( \frac{A}{b+A}\right) ^{p}\left( f\left(
t\right) +g\left( t\right) \right) ^{p}.
\end{equation}

Multiplying by $\displaystyle\frac{\rho ^{1-\beta }x^{\kappa }t^{\rho \left( \eta +1\right) -1}}{\Gamma\left( \alpha \right) \left( x^{\rho }-t^{\rho }\right) ^{1-\alpha }}$ both sides of {\rm Eq.(\ref{C3})} and integrating with respect to the variable $t$, we have
\begin{equation*}
\int_{a}^{x}\frac{\rho ^{1-\beta }x^{\kappa }t^{\rho \left( \eta +1\right)
-1}}{\Gamma \left( \alpha \right) \left( x^{\rho }-t^{\rho }\right)
^{1-\alpha }}g^{p}\left( t\right) dt\leq \int_{a}^{x}\frac{\rho ^{1-\beta
}x^{\kappa }t^{\rho \left( \eta +1\right) -1}}{\Gamma \left( \alpha \right)
\left( x^{\rho }-t^{\rho }\right) ^{1-\alpha }}\left( \frac{B}{a+B}\right)
^{p}\left( f\left( t\right) +g\left( t\right) \right) ^{p}dt.
\end{equation*}

Thus, it follows that
\begin{equation}\label{C5}
\left( ^{\rho }\mathcal{I}_{a+,\eta ,\kappa }^{\alpha ,\beta }g^{p}\left( x\right)
\right) ^{\frac{1}{p}}\leq \frac{B}{a+B}\left( ^{\rho }\mathcal{I}_{a+,\eta ,\kappa
}^{\alpha ,\beta }\left( f+g\right) ^{p}\left( x\right) \right) ^{\frac{1}{p}}.
\end{equation}

Similarly, we performe the calculations for {\rm Eq.(\ref{C4})}, we get
\begin{equation}\label{C6}
\left( ^{\rho }\mathcal{I}_{a+,\eta ,\kappa }^{\alpha ,\beta }f^{p}\left( x\right)
\right) ^{\frac{1}{p}}\leq \frac{A}{b+A}\left( ^{\rho }\mathcal{I}_{a+,\eta ,\kappa
}^{\alpha ,\beta }\left( f+g\right) ^{p}\left( x\right) \right) ^{\frac{1}{p}}.
\end{equation}

Adding {\rm Eq.(\ref{C5})} and {\rm Eq.(\ref{C6})}, we conclude that
\begin{equation*}
\left( ^{\rho }\mathcal{I}_{a+,\eta ,\kappa }^{\alpha ,\beta }f^{p}\left( x\right)
\right) ^{\frac{1}{p}}+\left( ^{\rho }\mathcal{I}_{a+,\eta ,\kappa }^{\alpha ,\beta
}g^{p}\left( x\right) \right) ^{\frac{1}{p}}\leq \frac{A\left( a+B\right)
+B\left( b+A\right) }{\left( a+B\right) \left( b+A\right) }\left( ^{\rho
}\mathcal{I}_{a+,\eta ,\kappa }^{\alpha ,\beta }\left( f+g\right) ^{p}\left( x\right)
\right) ^{\frac{1}{p}}.
\end{equation*}
\end{proof}

\begin{theorem} Let $\alpha>0$ and $\rho,\eta,\kappa,\beta\in\mathbb{R}$. Let $f,g\in X^{p}_{c}(a,x)$ be two positive functions in $[0,\infty)$, such that $\forall x>a$, $^{\rho }\mathcal{I}_{a+,\eta ,\kappa }^{\alpha ,\beta }f\left( x\right) <\infty $ and $^{\rho }\mathcal{I}_{a+,\eta ,\kappa }^{\alpha ,\beta }g\left( x\right) <\infty $. If $0<m\leq \displaystyle\frac{f\left( t\right) }{g\left( t\right) }\leq M$, for $m,M\in\mathbb{R}^{*}_{+}$ and $\forall t\in \left[ a,x\right] $, then 
\begin{eqnarray}
\frac{1}{M}\left( ^{\rho }\mathcal{I}_{a+,\eta ,\kappa }^{\alpha ,\beta }f\left(
x\right) g\left( x\right) \right)  &\leq &\frac{1}{\left( m+1\right) \left(
M+1\right) }\left( ^{\rho }\mathcal{I}_{a+,\eta ,\kappa }^{\alpha ,\beta }\left(
f+g\right) ^{2}\left( x\right) \right)   \notag \\
&\leq &\frac{1}{m}\left( ^{\rho }\mathcal{I}_{a+,\eta ,\kappa }^{\alpha ,\beta
}f\left( x\right) g\left( x\right) \right) .
\end{eqnarray}
\end{theorem}
\begin{proof}

Being $0<m\leq \displaystyle\frac{f\left( t\right) }{g\left( t\right) }\leq M$, $\forall t\in [a,x]$, we have
\begin{equation}\label{E1}
g\left( t\right) \left( m+1\right) \leq g\left( t\right) +f\left( t\right)
\leq g\left( t\right) \left( M+1\right). 
\end{equation}

Also, it follows that $\displaystyle\frac{1}{M}\leq \displaystyle\frac{g\left( t\right) }{f\left( t\right) }\leq \displaystyle\frac{1}{m}$, which implies,
\begin{equation}\label{E2}
g\left( t\right) \left( \frac{M+1}{M}\right) \leq g\left( t\right) +f\left(
t\right) \leq g\left( t\right) \left( \frac{m+1}{m}\right) .
\end{equation}

Evaluating the product between {\rm Eq.(\ref{E1})} and {\rm Eq.(\ref{E2})}, we have
\begin{equation}\label{E3}
\frac{f\left( t\right) g\left( t\right) }{M}\leq \frac{\left( g\left(
t\right) +f\left( t\right) \right) ^{2}}{\left( m+1\right) \left( M+1\right) 
}\leq \frac{f\left( t\right) g\left( t\right) }{m}.
\end{equation}

Multiplying by $\displaystyle\frac{\rho ^{1-\beta }x^{\kappa }t^{\rho \left( \eta +1\right) -1}}{\Gamma \left( \alpha \right) \left( x^{\rho }-t^{\rho }\right) ^{1-\alpha }}$ both sides of {\rm Eq.(\ref{E3})} and integrating with respect to the variable $t$, we have
\begin{eqnarray*}
\frac{\rho ^{1-\beta }x^{\kappa }}{M\Gamma \left( \alpha \right) }%
\int_{a}^{x}\frac{t^{\rho \left( \eta +1\right) -1}}{\left( x^{\rho
}-t^{\rho }\right) ^{1-\alpha }}f\left( t\right) g\left( t\right) dt &\leq
&c_{6}\frac{\rho ^{1-\beta }x^{\kappa }}{\Gamma \left( \alpha \right) }%
\int_{a}^{x}\frac{t^{\rho \left( \eta +1\right) -1}}{\left( x^{\rho
}-t^{\rho }\right) ^{1-\alpha }}\left( g\left( t\right) +f\left( t\right)
\right) ^{2}dt \\
&\leq &\frac{\rho ^{1-\beta }x^{\kappa }}{m\Gamma \left( \alpha \right) }%
\int_{a}^{x}\frac{t^{\rho \left( \eta +1\right) -1}}{\left( x^{\rho
}-t^{\rho }\right) ^{1-\alpha }}f\left( t\right) g\left( t\right) dt,
\end{eqnarray*}
with $c_{6}=\displaystyle\frac{1}{\left( m+1\right) \left( M+1\right) }$.

Thus, we conclude that
\begin{eqnarray*}
\frac{1}{M}\left( ^{\rho }\mathcal{I}_{a+,\eta ,\kappa }^{\alpha ,\beta }f\left(
x\right) g\left( x\right) \right)  &\leq &\frac{1}{\left( m+1\right) \left(
M+1\right) }\left( ^{\rho }\mathcal{I}_{a+,\eta ,\kappa }^{\alpha ,\beta }\left(
g\left( x\right) +f\left( x\right) \right) ^{2}\right)  \\
&\leq &\frac{1}{m}\left( ^{\rho }\mathcal{I}_{a+,\eta ,\kappa }^{\alpha ,\beta
}f\left( x\right) g\left( x\right) \right). 
\end{eqnarray*}
\end{proof}

\begin{theorem} Let $\alpha>0$, $\rho,\eta,\kappa,\beta\in\mathbb{R}$ and $p\geq 1$. Let $f,g\in X^{p}_{c}(a,x)$ be two positive functions in $[0,\infty)$, such that $\forall x>a$, $^{\rho }\mathcal{I}_{a+,\eta ,\kappa }^{\alpha ,\beta }f^{p}\left( x\right) <\infty $ and $^{\rho }\mathcal{I}_{a+,\eta ,\kappa }^{\alpha ,\beta }g^{p}\left( x\right) <\infty $. If $0<m\leq \displaystyle\frac{f\left( t\right) }{g\left( t\right) }\leq M$, for $m,M\in\mathbb{R}^{*}_{+}$ and $\forall t\in \left[ a,x\right] $, then
\begin{equation}
\left( ^{\rho }\mathcal{I}_{a+,\eta ,\kappa }^{\alpha ,\beta }f^{p}\left( x\right)
\right) ^{\frac{1}{p}}+\left( ^{\rho }\mathcal{I}_{a+,\eta ,\kappa }^{\alpha ,\beta
}g^{p}\left( x\right) \right) ^{\frac{1}{p}}\leq 2\left( ^{\rho }\mathcal{I}_{a+,\eta
,\kappa }^{\alpha ,\beta }h^{p}\left( f\left( x\right) ,g\left( x\right)
\right) \right) ^{\frac{1}{p}},  \notag
\end{equation}
with $h\left( f\left( x\right) ,g\left( x\right) \right) =\max \left\{ M\left[
\left( \displaystyle\frac{M}{m}+1\right) f\left( x\right) -Mg\left( x\right) \right] ,\displaystyle\frac{\left( m+M\right) g\left( x\right) -f\left( x\right) }{m}\right\} $.
\end{theorem}
\begin{proof} From the hypothesis, $0<m\leq \displaystyle\frac{f\left( t\right) }{g\left( t\right) }\leq M$, $\forall t\in [a,x]$, we have
\begin{equation}\label{F1}
0<m\leq M+m-\frac{f\left( t\right) }{g\left( t\right) }
\end{equation}
and
\begin{equation}\label{F2}
M+m-\frac{f\left( t\right) }{g\left( t\right) }\leq M.
\end{equation}

Thus, using {\rm Eq.(\ref{F1})} and {\rm Eq.(\ref{F2})}, we get
\begin{equation}\label{F3}
g\left( t\right) <\frac{\left( M+m\right) g\left( t\right) -f\left( t\right) 
}{m}\leq h\left( f\left( t\right) ,g\left( t\right) \right) ,
\end{equation}
where $h\left( f\left( t\right) ,g\left( t\right) \right) =\max \left\{ M\left[
\left( \frac{M}{m}+1\right) f\left( t\right) -Mg\left( t\right) \right] ,%
\frac{\left( M+m\right) g\left( t\right) -f\left( t\right) }{m}\right\}. $

Using the hypothesis, it follows that $0<\displaystyle\frac{1}{M}\leq\displaystyle \frac{g\left( t\right) }{f\left( t\right) }\leq \displaystyle\frac{1}{m}$. In this way, we obtain
\begin{equation}\label{F4}
\frac{1}{M}\leq \frac{1}{M}+\frac{1}{m}-\frac{g\left( t\right) }{f\left(
t\right) }
\end{equation}
and
\begin{equation}\label{F5}
\frac{1}{M}+\frac{1}{m}-\frac{g\left( t\right) }{f\left( t\right) }\leq 
\frac{1}{m}.
\end{equation}

Then, from {\rm Eq.(\ref{F4})} and {\rm Eq.(\ref{F5})}, we have
\begin{equation*}
\frac{1}{M}\leq \frac{\left( \frac{1}{m}+\frac{1}{M}\right) f\left( t\right)
-g\left( t\right) }{f\left( t\right) }\leq \frac{1}{m},
\end{equation*}

which can be rewrite as
\begin{eqnarray}\label{F6}
f\left( t\right)  &\leq &M\left( \frac{1}{m}+\frac{1}{M}\right) f\left(
t\right) -Mg\left( t\right)   \notag \\
&=&\frac{M\left( M+m\right) f\left( t\right) -M^{2}mg\left( t\right) }{Mm} 
\notag \\
&=&\left( \frac{M}{m}+1\right) f\left( t\right) -Mg\left( t\right)   \notag
\\
&\leq &M\left[ \left( \frac{M}{m}+1\right) f\left( t\right) -Mg\left(
t\right) \right]   \notag \\
&\leq &h\left( f\left( t\right) ,g\left( t\right) \right) .
\end{eqnarray}

Thus, using {\rm Eq.(\ref{F3})} and {\rm Eq.(\ref{F6})}, we can write
\begin{equation}\label{F7}
f^{p}\left( t\right) \leq h^{p}\left( f\left( t\right) ,g\left( t\right)\right) 
\end{equation}
and
\begin{equation}\label{F8}
g^{p}\left( t\right) \leq h^{p}\left( f\left( t\right) ,g\left( t\right)\right). 
\end{equation}

Multiplying by $\displaystyle\frac{\rho ^{1-\alpha }x^{\kappa }t^{\rho \left( \eta +1\right) -1}}{\Gamma \left( \alpha \right) \left( x^{\rho }-t^{\rho }\right) ^{1-\alpha }}$ both sides of {\rm Eq.(\ref{F7})} and integrating with respect to the variable $t$, we have
\begin{equation*}
\int_{a}^{x}\frac{\rho ^{1-\alpha }x^{\kappa }t^{\rho \left( \eta +1\right)
-1}}{\Gamma \left( \alpha \right) \left( x^{\rho }-t^{\rho }\right)
^{1-\alpha }}f^{p}\left( t\right) dt\leq \int_{a}^{x}\frac{\rho ^{1-\alpha
}x^{\kappa }t^{\rho \left( \eta +1\right) -1}}{\Gamma \left( \alpha \right)
\left( x^{\rho }-t^{\rho }\right) ^{1-\alpha }}h^{p}\left( f\left( t\right)
,g\left( t\right) \right) dt.
\end{equation*}

In this way, we obtain
\begin{equation}\label{F9}
\left( ^{\rho }\mathcal{I}_{a+,\eta ,\kappa }^{\alpha ,\beta }f^{p}\left( x\right)
\right) ^{\frac{1}{p}}\leq \left( ^{\rho }\mathcal{I}_{a+,\eta ,\kappa }^{\alpha
,\beta }h^{p}\left( f\left( x\right) ,g\left( x\right) \right) \right) ^{%
\frac{1}{p}}.
\end{equation}

Using the same procedure as above, for {\rm Eq.(\ref{F8})}, we have
\begin{equation}\label{F10}
\left( ^{\rho }\mathcal{I}_{a+,\eta ,\kappa }^{\alpha ,\beta }g^{p}\left( x\right)
\right) ^{\frac{1}{p}}\leq \left( ^{\rho }\mathcal{I}_{a+,\eta ,\kappa }^{\alpha
,\beta }h^{p}\left( f\left( x\right) ,g\left( x\right) \right) \right) ^{%
\frac{1}{p}}.
\end{equation}

Thus, using {\rm Eq.(\ref{F9})} and {\rm Eq.(\ref{F10})}, we conclude that
\begin{equation*}
\left( ^{\rho }\mathcal{I}_{a+,\eta ,\kappa }^{\alpha ,\beta }f^{p}\left( x\right)
\right) ^{\frac{1}{p}}+\left( ^{\rho }\mathcal{I}_{a+,\eta ,\kappa }^{\alpha ,\beta
}g^{p}\left( x\right) \right) ^{\frac{1}{p}}\leq 2\left( ^{\rho }I_{a+,\eta
,\kappa }^{\alpha ,\beta }h^{p}\left( f\left( x\right) ,g\left( x\right)
\right) \right) ^{\frac{1}{p}}.
\end{equation*}

\end{proof}

Using Eq.(\ref{A19}) and Theorem 7 with the convenient conditions for each respective fractional integral, we have the previous theorems, that is, Theorem 10 to Theorem 15 introduced and demonstrated above, contain as particular cases, each result involving the following fractional integrals: Riemann-Liouville, Hadamard, Liouville, Weyl, Edérlyi-Kober, and Katugampola.

\section{Concluding remarks}

After a brief introduction to the fractional integral, proposed by Katugampola and fractional integrals in the sense of Riemann-Liouville and Hadamard, we generalize the reverse Minkowski's inequality obtaining, as a particular case, the inequality involving the fractional integral in the Riemann-Liouville sense and Hadamard sense \cite{UNT}. We also show other inequalities using the Katugampola fractional integral. The application of this fractional integral can be used to generalize several inequalities, among them, we mention the Gruss-type inequality, recently introduced and proved \cite{JDE}. A continuation of this work, with this formulation of fractional integral, consists in generalize the inequalities of Hermite-Hadamard and Hermite-Hadamard-Féjer. Moreover, we will discuss inequalities via $M$-fractional integral according to \cite{JEC}.

\bibliography{ref}
\bibliographystyle{plain}

\end{document}